\documentclass[leqno]{aadmbook}
\usepackage{amsthm}
\usepackage{amsfonts}
\usepackage{amsmath}
\usepackage[english]{babel}
\usepackage[T1]{fontenc}

\usepackage{amsmath}
\usepackage{amsfonts}
\usepackage{amsthm}
\usepackage{amscd,latexsym,amssymb}
\usepackage{graphicx}
\usepackage{url}
\usepackage{color}
\usepackage{enumerate}

\newtheorem{rem}{Remark}
\newtheorem{obs}{Observation}
\newtheorem{thm}{Theorem}%[section]
\newtheorem{prop}{Proposition}%[section]
\newtheorem{cor}{Corollary}%[section]
\newtheorem{lem}{Lemma}%[section]
\begin{document}

% Your \newcommands below (if there are any):

\newcommand{\coro}{\textrm{cor}}

\oddsidemargin 16.5mm
\evensidemargin 16.5mm

\thispagestyle{plain}

%\begin{center}
%{\large \sc  Applicable Analysis and Discrete Mathematics}

%{\small available online at  http:/$\!$/pefmath.etf.rs }
%\end{center}

%\noindent{\small{\sc  Appl. Anal. Discrete Math.\ }{\bf x} (xxxx),
%xxx--xxx.} \hfill{\scriptsize doi:10.2298/AADMxxxxxxxx}

%\vspace{5cc}
\begin{center}

{\large\bf  PERFECT AND QUASIPERFECT DOMINATION IN TREES
\rule{0mm}{6mm}\renewcommand{\thefootnote}{}%Enter at least one, but not more than 3 MSCs.
% First entered MSC will be a primary one, others (at most 2) will be secondary.
\footnotetext{

\rule{2.4mm}{0mm}Keywords and Phrases. Domination, perfect domination, quasiperfect domination,
trees.
}}

\vspace{1cc}
{\large\it Jos\'e C\'aceres, Carmen Hernando, Merc\`{e} Mora, Ignacio M. Pelayo and Mar\'ia Luz Puertas}

\vspace{1cc}
\parbox{24cc}{{\small

A $k-$quasiperfect dominating set ($k\ge 1$) of a graph $G$ is a vertex subset $S$ such that every vertex not
in $S$ is adjacent to at least one and at most k vertices in $S$. The cardinality of a minimum k-quasiperfect
dominating set in $G$ is denoted by $\gamma_{\stackrel{}{1k}}(G)$. Those sets were first
introduced by Chellali et al. (2013) as a generalization of the perfect domination concept. The quasiperfect domination chain $\gamma_{\stackrel{}{11}}(G)\ge\gamma_{\stackrel{}{12}}(G)\ge\dots\ge\gamma_{\stackrel{}{1\Delta}}(G)=\gamma(G)$, indicates what it is lost in size when you move towards a more perfect domination. We provide an upper bound for $\gamma_{\stackrel{}{1k}}(T)$ in any tree $T$ and trees achieving this bound are characterized. We prove that there exist trees satisfying all the possible equalities and inequalities in this chain and a linear algorithm for computing $\gamma_{\stackrel{}{1k}}(T)$ in any tree is presented.
}}
\end{center}

\vspace{1cc}

\vspace{1.5cc}
\begin{center}
{\bf 1. INTRODUCTION}
\end{center}
\vspace{1cc}

All the graphs considered here are finite, undirected, simple, and connected. For undefined basic concepts we refer the reader to introductory graph theoretical literature as~\cite{chlepi11}. Recall that a \emph{tree} is a connected acyclic graph. A \emph{leaf} is a vertex of degree 1 and vertices of degree at least $2$ are \emph{interior} vertices. A \emph{support vertex} is a vertex having at least one leaf in its neighborhood and a \emph{strong support vertex} is a support vertex adjacent to at least two leaves.

Given a graph $G$, a subset $S$ of its vertices is a \emph{dominating set} of $G$ if every vertex $v$ not in $S$ is adjacent to at least one vertex in $S$. The \emph{domination number} $\gamma(G)$ is the minimum cardinality of a dominating set of $G$, and a dominating set of cardinality $\gamma(G)$ is called a \emph{$\gamma$-code}~\cite{hahesl}.

An extreme way of domination occurs when every vertex not in $S$ is adjacent to exactly one vertex in $S$. In that case, $S$ is called a \emph{perfect dominating set}~\cite{cohahela93} and $\gamma_ {\stackrel{}{11}}(G)$, the minimum cardinality of a perfect dominating set of $G$, is the \emph{perfect domination number}. A dominating set of cardinality $\gamma_ {\stackrel{}{11}}(G)$ is called a \emph{$\gamma_ {\stackrel{}{11}}$-code}. This concept has been studied in~\cite{dej08,dejdel09,kwle,listo90}.

In a perfect dominating set what it is gained from the point of view of prefection it is lost in size, comparing it with a dominating set. Between both notions there is a graduation of definitions given by the so-called $k$-quasiperfect domination. A \emph{$k$-quasiperfect dominating set} for $k\geq 1$ (\emph{$\gamma_ {\stackrel{}{1k}}$-set} for short)~\cite{chhahemc13,yang} is a dominating set $S$ such that every vertex not in $S$ is adjacent to at most $k$ vertices of $S$. Again the \emph{$k$-quasiperfect domination number} $\gamma_ {\stackrel{}{1k}}(G)$ is the minimum cardinality of a $\gamma_ {\stackrel{}{1k}}$-set of $G$ and a \emph{$\gamma_ {\stackrel{}{1k}}$-code} is a $\gamma_ {\stackrel{}{1k}}$-set of cardinality $\gamma_ {\stackrel{}{1k}}(G)$. This problem has been recently studied in \cite{bidughpa14,xuba15}.

Given a graph $G$ of order $n$ and maximum degree $\Delta$, $\gamma_ {\stackrel{}{1\Delta}}$-sets are precisely dominating sets. Thus, one can construct the following chain of quasiperfect domination parameters that we will call the \emph{QP-chain} of $G$:

\begin{equation*}
n \ge \gamma_ {\stackrel{}{11}}(G) \ge \gamma_ {\stackrel{}{12}}(G)\ge \ldots \ge \gamma_ {\stackrel{}{1\Delta}}(G)=\gamma(G)
\end{equation*}

We began the study of the QP-chain in \cite{cahemopepu} and now our attention is focused in the behavior of these parameters in the particular case of trees. This paper is organized as follows. In the next Section basic and known results about quasiperfect parameters are recalled. In Section 3 a general upper bound for the quasiperfect domination number in terms of the domination number is obtained. Section 4 is devoted to study the QP-chain, introducing a realization-type theorem for it. Finally, in Section 5 we provide an algorithm to compute the $k$-quasiperfect domination number of a tree in linear time.

%%%% known results of graphs and trees

\vspace{1.5cc}
\begin{center}
{\bf 2. BASIC AND GENERAL RESULTS}
\end{center}
\vspace{1cc}

In this Section we recall some known results about domination and perfect domination. The \emph{corona} of a graph $G$, denoted by $\coro (G)$, is the graph obtained by attaching a leaf to each vertex of $G$. It is well known that corona graphs achieve the maximum value of domination number.

\begin{thm} \cite{fijakiro85,paxu82}
For any graph $G$ the domination number satisfies $\gamma(G)\le n/2$. Moreover if $G$ is a graph of even order $n$, then
$\gamma (G)=n/2$ if and only if $G$ is the cycle of order 4 or the corona of a connected graph.
\end{thm}

Graphs with odd order $n$ and maximum domination number $\gamma(G)=\lfloor n/2 \rfloor$ are also completely characterized in \cite{bacohahesh}, as a list of six graph classes.

On the other hand, it is clear that for every graph $G$ of order $n\ge 3$ with $n_1$ vertices of degree 1, $\gamma_{\stackrel{}{11}} (G) \le n- n_1$, since the set of all vertices that are no leaves is a $\gamma_{\stackrel{}{11}}$-set. This property leads to the following observations for trees.

\begin{obs}
If $T$ is a tree of order at least $3$, there exists a $\gamma$-code containing no leaves, since the set obtained by removing a leaf and adding its support vertex, if necessary, is also a dominating set.
\end{obs}

\begin{obs}
Any $\gamma$-code of a tree contains all its strong support vertices. Suppose on the contrary that $v$ is a strong support vertex not in a $\gamma$-code $S$, then $S$ must contain at least two leaves adjacent to $v$, but the set $(S\setminus \{x,y \}) \cup \{ v \}$ is a dominating set with less vertices than $S$, which is not possible.
\end{obs}

Similar results are known for the perfect domination number of trees.

\begin{prop} \cite{cahahe12} Let $T$ be a tree of order $n\ge 3$. Then
\begin{enumerate}
  \item  Every $\gamma_{\stackrel{}{11}}-code$ of $T$ contains all its strong support vertices.
  \item  $\gamma_{\stackrel{}{11}} (T)\le n/2$.
  \item  $\gamma_{\stackrel{}{11}} (T) = n/2$ if and only if $T=\coro (T')$ for some tree $T'$.
\end{enumerate}
\end{prop}

The following corollary is a consequence of the preceding results.

\begin{cor}\label{equality}
  If $T$ is a tree of order $n\ge 3$, the following conditions are equivalent:
  \begin{enumerate}
    \item $\gamma (T)= n/2$.
    \item $\gamma_{\stackrel{}{11}} (T)= n/2$.
    \item $T=\coro (T')$, for some  tree $T'$.
  \end{enumerate}
\end{cor}

This corollary shows that the QP-chain adopts its shortest form in graphs which are the corona of a tree, for instance in the comb graph $\coro(P_m)$ with $m\geq 3$ which is the corona of the path $P_{m}$. In this case $n/2=m=\gamma_{\stackrel{}{11}}(\coro(P_m))=\gamma(\coro (P_m))$. There are other simple tree families having a constant QP-chain. For instance any path $P_n$ satisfies
$$\gamma_{\stackrel{}{11}}(P_n)=\gamma_{\stackrel{}{12}} (P_n) =\gamma (P_n)=\lceil n/3 \rceil.$$
A star $K_{1,n-1}$ with $n$ vertices and maximum degree $n-1$, satisfies $$\gamma_{\stackrel{}{11}} (K_{1,n-1})=\gamma_{\stackrel{}{12}} (K_{1,n-1}) =\dots =\gamma_{\stackrel{}{1, n-1}}(K_{1,n-1})=\gamma (K_{1,n-1})=1.$$

Finally, recall that a \emph{caterpillar} is a tree that has a dominating path. This special class of trees has a particular behavior regarding the QP-chain.

\begin{prop} \cite{chhahemc13}\label{equality_caterpillar}
If $T$ is a caterpillar, then $\gamma (T)=\gamma_{\stackrel{}{12}} (T)$.
\end{prop}

\newpage
%\vspace{1.5cc}
\begin{center}
{\bf 3. BOUNDS FOR QUASIPERFECT DOMINATION IN TREES}
\end{center}
\vspace{1cc}

\noindent{\bf 3.1 General upper bound}
\vspace{1cc}

The QP-chain shows that the domination number $\gamma(T)$ is a natural lower bound of the quasiperfect domination number $\gamma_{\stackrel{}{1k}}(T)$. Furthermore, this bound can be reached, for instance the path $P_n$ satisfies $\gamma_{\stackrel{}{11}}(P_n)=\gamma_{\stackrel{}{12}} (P_n) =\gamma (P_n)$ and star $K_{1,r}, r\geq 2$ satisfies $\gamma_{\stackrel{}{1k}}(K_{1,r})=\gamma (K_{1,r}), 1\le k\le r$.
Our main result in this Section provides an upper bound of quasi-perfect domination numbers of a tree in terms of the domination number.

If $S\subseteq V(T)$, we denote by $T[S]$ the subgraph of $T$ induced by the vertices of $S$.

\begin{lem}\label{lem.components}
Let $T$ be a tree and let $S$ be a dominating set. Then, every vertex not in $S$ has at most one neighbor at each connected component of the subgraph $T[S]$.
\end{lem}
\begin{proof}
If a vertex not in $S$ has two neighbors in a connected component of $T[S]$ then $T$ has a cycle, which is not possible.
\end{proof}

As a consequence, the following result is obtained:

\begin{cor}\label{cor.components}
Let $T$ be a tree and $S$ a dominating set such that the subgraph $T[S]$ has at most $k$ connected components, then $S$ is a $\gamma_{\stackrel{}{1k}}$-set.
\end{cor}

\begin{thm}\label{thm.upperbound}
For every tree $T$ and for every integer $k\geq 1$,
\begin{equation}\label{eq:generalbound}
 \gamma_{\stackrel{}{1k}}(T)\le  \gamma (T) +\bigg\lceil \frac{\gamma (T)}{k}\bigg\rceil-1
 \end{equation}
and this bound is tight.
\end{thm}

\begin{proof}
Let $S$ be a $\gamma$-code of $T$. If $S$ is a $\gamma_{\stackrel{}{1k}}$-set, then inequality~(\ref{eq:generalbound}) trivially holds.

Suppose on the contrary that $S$ is not a $\gamma_{\stackrel{}{1k}}$-set. We intend to construct a $\gamma_{\stackrel{}{1k}}$-set $S^*$ containing $S$  and satisfying the desired inequality. Let $r$ be the number of connected components of the subgraph $T[S]$. Then, $\gamma(T)\ge r$ and, by Corollary~\ref{cor.components}, $r>k$.

Consider a vertex $x_0\in V(T)\setminus S$ with at least $k+1$ neighbors in $S$ and let $S_1=S\cup \{ x_0 \}$. By Lemma~\ref{lem.components}, all the neighbors of $x_0$ in $S$ lie in
different connected components of $T[S]$, therefore $S_1$ is a dominating set inducing a subgraph $T[S_1]$ with at most $r-k$ connected components. If $S_1$ is a $\gamma_{\stackrel{}{1k}}$-set, let $S^*=S_1$. Otherwise, consider a vertex $x_{1}\in V(T)\setminus S_1$ having at least $k+1$ neighbors in $S_1$ and let $S_2=S_1\cup \{ x_1\}$. Again all the neighbors of $x_1$ in $S_1$ lie in different connected components of $T[S_1]$, therefore $S_2$ is a dominating set inducing a subgraph $T[S_2]$ with at most $(r-k)-k=r-2k$ connected components. If $S_2$ is a $\gamma_{\stackrel{}{1k}}$-set, let $S^*=S_2$.

Observe that this proceeding will end since $T[S_i]$ has at most $r-ik$ connected components, and this number sequence is strictly decreasing.
In the worst case, you should consider $j=\lceil \frac {r-k}{k} \rceil$ with  $S_j$ having at most $r-jk$ connected components because in this case $r-jk\leq k$ and $S_j$ must be a $\gamma_{\stackrel{}{1k}}$-set. So $|S^*|\leq |S|+j=\gamma(T)+\lceil \frac {r-k}{k} \rceil$. Hence

$$\gamma_{\stackrel{}{1k}}(T)\le |S^*|\leq |S|+j= \gamma(T)+ \bigg\lceil \frac {r-k}{k} \bigg\rceil
\le \gamma(T)+ \bigg\lceil \frac {\gamma(T)-k}{k} \bigg\rceil=\gamma(T)+ \bigg\lceil \frac {\gamma(T)}{k} \bigg\rceil-1 .$$

We finally show the tightness of the bound. Notice that if $k\geq \gamma (T)$ then $\gamma_{\stackrel{}{1k}}(T)=\gamma(T)$ and $\bigg\lceil \displaystyle\frac {\gamma(T)}{k} \bigg\rceil=1$, so in this case $\gamma_{\stackrel{}{1k}}(T)=\gamma(T)+\bigg\lceil \displaystyle\frac {\gamma(T)}{k} \bigg\rceil-1$.

Next, suppose that $\gamma(T)=a,\ a\geq 2$ and $k<a$. Consider the graph in Figure~\ref{fig.realization.upper.bound} where $a=q\cdot k+r,\ q\geq 1, 1\leq r\leq k$. It is clear that the set of squared vertices is a $\gamma$-code so $\gamma(T)=a$ and the set of black vertices is a $\gamma_{\stackrel{}{1k}}$-code so

$\gamma(T)+ \bigg\lceil \displaystyle\frac {\gamma (T)}{k} \bigg\rceil-1 =a+\bigg\lceil \frac{q\cdot k+r}{k}\bigg\rceil-1 =a+q+1-1=q\cdot k+r +q =\gamma_{\stackrel{}{1k}}(T)$.
\end{proof}

\begin{figure}[!hbt]
\begin{center}
\includegraphics[width=0.9\textwidth]{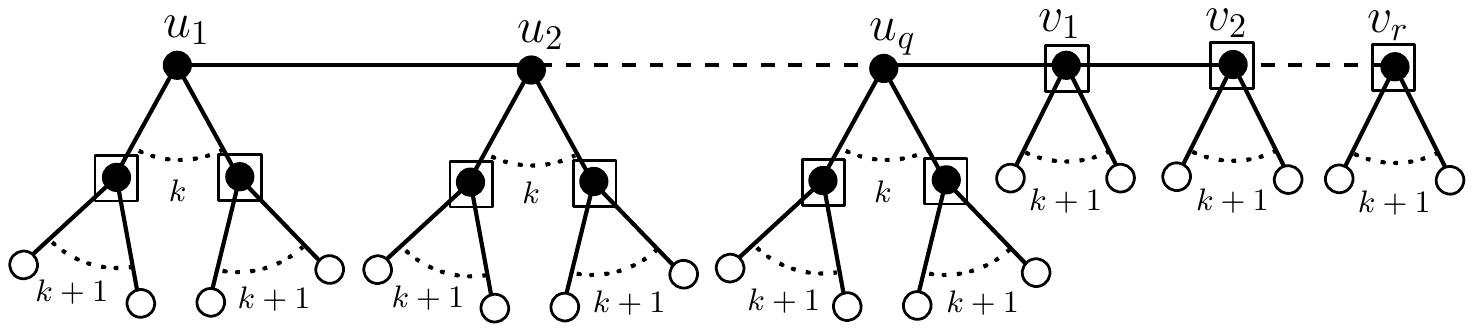}
\caption{Squared vertices are a $\gamma$-code and black vertices are a $\gamma_{\stackrel{}{1k}}$-code}\label{fig.realization.upper.bound}
\end{center}
\end{figure}

\vspace{1.5cc}
\noindent{\bf 3.2 Trees satisfying $\gamma_{\stackrel{}{11}}(T)=2\gamma(T)-1$ }
\vspace{1cc}

In the particular case of the perfect domination number the upper bound shown in Theorem~\ref{thm.upperbound} is the following

$$\displaystyle \gamma_{\stackrel{}{11}}(T)\le  \gamma (T) +\bigg\lceil \frac{\gamma (T)}{1}\bigg\rceil -1 = 2\gamma (T) -1$$

Notice that this bound is far from being true for general graphs and and, as a matter of fact, the difference between both parameters can be as large as desired. For instance, the graph
shown in Figure~\ref{fig.noboundgraph} satisfies $\gamma (G)=2$ and $\gamma_{\stackrel{}{11}}(G)=|V(G)|>2\gamma(G)-1$.

\begin{figure}[!hbt]
\begin{center}
\includegraphics[width=0.22\textwidth]{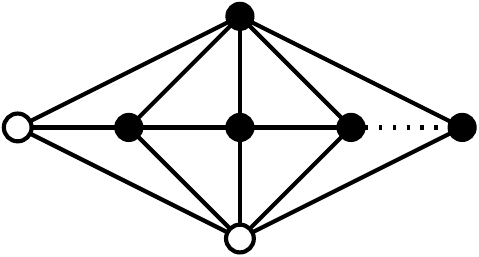}
\caption{The pair of white
vertices form a $\gamma$-code meanwhile $\gamma_{\stackrel{}{11}}(G)=|V(G)|$.}\label{fig.noboundgraph}
\end{center}
\end{figure}

Let $T$ be a tree satisfying $\gamma_{\stackrel{}{11}}(T)=2\gamma(T)-1$ then for any $\gamma$-code $S$ of $T$ the associated $\gamma_{\stackrel{}{11}}$-set $S^*$ constructed in Theorem~\ref{thm.upperbound} satisfies $|S^*|=2\gamma(T)-1$, so it is also a $\gamma_{\stackrel{}{11}}$-code. However, some trees contain $\gamma_{\stackrel{}{11}}$-codes which can no be obtained from this construction. For instance, the tree shown in Figure~\ref{fig.boundtree} has a $\gamma_{\stackrel{}{11}}$-code which does not contain any $\gamma$-code.

\begin{figure}[!hbt]
\begin{center}
\includegraphics[width=0.35\textwidth]{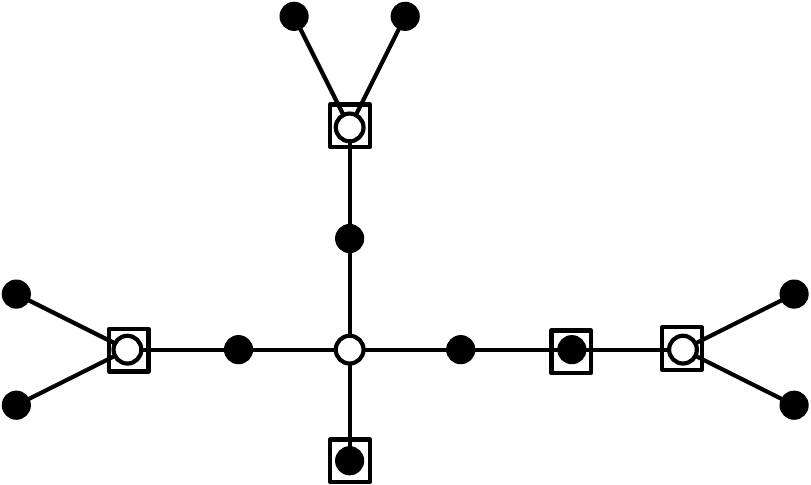}
\caption{Squared vertices form the unique $\gamma_{\stackrel{}{11}}$-code and they do not contain the unique $\gamma$-code consisting on white vertices.}\label{fig.boundtree}
\end{center}
\end{figure}

Our next goal is to characterize the family of trees achieving this bound. To this end we review the construction of the $\gamma_{\stackrel{}{11}}$-set associated to a $\gamma$-code given in Theorem~\ref{thm.upperbound}. Let $S$ be a $\gamma$-code of a tree $T$ which is not a $\gamma_{\stackrel{}{11}}$-code. Notice that since $S$ is not a $\gamma_{\stackrel{}{11}}$-set there exists at least one vertex $x\notin S$ that is not a leaf. Denote by $C_1,\dots,C_k$, $k\ge 1$, the connected components of the graph $T-(S\cup L')$, where $L'$ is the set of leaves of $T$ not in $S$. Then for some $r\in \{1,\dots,k\}$, at least one vertex in each $C_1,\dots ,C_r$ has two or more neighbors in $S$ and vertices in $C_{r+1},\dots ,C_k$ (if $r<k$) have exactly one neighbor in $S$. In Proposition below we follow this notation and a precise description of the $\gamma_{\stackrel{}{11}}$-set $S^*$ associated to $S$ is provided.

\begin{prop}\label{prop.properties.S^*}
Let $S$ be a $\gamma$-code of a tree $T$ which is not a $\gamma_{\stackrel{}{11}}$-set. Then $S^*=S\cup (\bigcup_{i=1}^r V(C_i))$ has the following properties.
\begin{enumerate}
  \item $S^*$ is a $\gamma_{\stackrel{}{11}}$-set of $T$.
  \item $S^*$ has at most $2 \gamma(T)-1$ vertices.
  \item If $S'$ is a $\gamma_{\stackrel{}{11}}$-set of $T$ containing $S$ then $S^*\subseteq S'$.
\end{enumerate}
\end{prop}

\begin{proof}
\begin{enumerate}
\item  Let $u\in V(T)\setminus S^*$. It $u$ is a leaf then it has just one neighbor in $S^*$. Suppose now that $u\notin S^*\cup L'$, then there exists $i\in \{r+1,\dots k\}$ such that $u\in V(C_i)$ and it has just one neighbor in $S$. Using that the connected components are pairwise disjoint, it is clear that $u$ has exactly one neighbor in $S^*$ as desired.

\item Consider the tree $T-L'$. By construction,
$V(T-L')=S\cup V(C_1)\cup \dots \cup V(C_k)$, where $S,V(C_1), \dots ,V(C_k)$ are pairwise disjoint sets. Therefore,
\begin{equation}\label{eq.leaves}
|E(T-L')|=|V(T-L')|-1=|S|+\sum_{i=1}^{k}|V(C_i)|-1.
\end{equation}

Now observe that the edges of $T-L'$ connect two vertices of one of the connected components $C_i$, or two vertices of the $\gamma$-code $S$, or a vertex of $S$ with a vertex of some $C_i$.
For any pair of subsets of vertices $A$, $B$, let us denote $E(A:B)$  the set of edges with an endpoint in $A$ and the other one in $B$.
With this notation, we have that:
$$E(T-L')= \Big(\bigcup_{i=1}^k E(C_i:S)\Big)\cup \Big(\bigcup_{i=1}^k E(C_i:C_i)\Big)\cup E(S:S).$$

Moreover, the $2k+1$ subsets involved in this union are pairwise disjoint.

By hypotheses $|E(C_i:S)|=|V(C_i)|+\delta_i$, where $\delta_i\ge 1$ for all $i\in \{1,\dots ,r\}$, and $|E(C_i:S)|=|V(C_i)|$ for all $i\in \{r+1,\dots ,k\}$.
On the other hand, $|E(C_i:C_i)|=|V(C_i)|-1$ for all $i\in \{ 1,\dots ,k\}$, using that each $C_i$ is a tree.
From these observations we obtain
\begin{equation}\label{eq.sum}
\begin{split}
|E(T-L')|&=\sum_{i=1}^{k}|E(C_i:S)|+\sum_{i=1}^{k}|E(C_i:C_i)|+|E(S:S)|\\
&=\sum_{i=1}^{k}(|V(C_i)|-1)+\sum_{i=1}^{r}|E(C_i:S)|+\sum_{i=r+1}^{k}|E(C_i:S)|+|E(S:S)|\\
&= \sum_{i=1}^{k}|V(C_i)|-k+\sum_{i=1}^{r}(|V(C_i)|+\delta_i)+\sum_{i=r+1}^{k}|V(C_i)|+|E(S:S)|\\
&= \sum_{i=1}^{k}|V(C_i)|-k+\sum_{i=1}^{r}|V(C_i)|+\sum_{i=1}^{r}\delta_i+\sum_{i=r+1}^{k}|V(C_i)|+|E(S:S)|\\
&= \sum_{i=1}^{k}|V(C_i)|-k+|S^*|-|S|+\sum_{i=1}^{r}\delta_i+\sum_{i=r+1}^{k}|V(C_i)|+|E(S:S)|
\end{split}
\end{equation}

From Equations~\ref{eq.leaves} and~\ref{eq.sum} and using that $|V(C_i)|\ge 2$ for all $i\in \{r+1,\dots ,k\}$, because as otherwise the unique vertex in $C_i$ would be a leaf, we obtain
\begin{equation*}\label{eq.final}
\begin{split}
2|S|-1=
&|S^*|+\sum_{i=1}^{r}\delta_i+\sum_{i=r+1}^{k}|V(C_i)|-k+|E(S:S)|\\
&\ge |S^*|+r+2(k-r)-k\\
&=|S^*|+k-r\\
&\ge |S^*|.\end{split}
\end{equation*}

\item Let $S'$ be a $\gamma_{\stackrel{}{11}}$-set of $T$ containing $S$ and suppose on the contrary that $V(C_i)\setminus S'\neq \emptyset$ for some $i\in \{1,\dots ,r \}$. Let $v\in V(C_i)\setminus S'$ and let $u_i\in V(C_i)$ be a vertex with at least two neighbors in $S$. It is clear that $u_i\in S'$. Consider a $u_i-v$ path $P$ in $C_i$ and let $w$ be the first vertex of the path not in $S'$. Then, $w$ has at least one neighbor in $S\subseteq S'$ and a neighbor in $S'\cap V(P)\subseteq S'\setminus S$, contradicting the fact that $S'$ is a $\gamma_{\stackrel{}{11}}$-set.
\end{enumerate}
\end{proof}

Now we present some properties involving $\gamma$-codes and its associated $\gamma_{\stackrel{}{11}}$-sets, when the upper bound is reached. For a vertex set $C$ we denote by $N(C)$ the set of all neighbors of the vertices of $C$. We also denote by $L$ the set of leaves of $T$.

\begin{lem}\label{lem.estructure.S^*}
Let $T$ be a tree such that $\gamma_{\stackrel{}{11}}(T)=2 \gamma(T)-1$. Let $S$ be a $\gamma$-code of $T$ and let $L'$ be the set of leaves not in $S$. Then

\begin{enumerate}

\item $S$ is an independent set and every connected component $C_i$ of $T-(S\cup L')$ satisfies $|N(V(C_i))\cap S|=|V(C_i)|+1$.

\item $S^*=V(T)\setminus L'$. Moreover, if $S$ does not contain leaves, then $S^*=V(T)\setminus L$.
\end{enumerate}
\end{lem}

\begin{proof}
\begin{enumerate}
\item If $\gamma_{\stackrel{}{11}}(T)=2 \gamma(T)-1$ then $S^*$ is a $\gamma_{\stackrel{}{11}}$-code and $|S^*|=2|S|-1$. From Equation~\ref{eq.final}, we deduce that $|E(S:S)|=0$, $r=k$ and $\delta_i=1$ for every $i\in  \{1,\dots ,r\}$.
Therefore, $S$ is an independent set and for any $i\in \{ 1,\dots ,k\}$, $|E(C_i:S)|=|V(C_i)|+\delta_i= |V(C_i)|+1.$ Since two different vertices of the same connected component $C_i$ have no common neighbor in $S$, we obtain that $|N(V(C_i))\cap S|=|E(C_i:S)|=|V(C_i)|+1$.

\item It is a direct consequence of the construction of $S^*$ and the preceding item.

\end{enumerate}

\end{proof}

\begin{rem}
Condition 1 in the Lemma~\ref{lem.estructure.S^*} means that there exists exactly one vertex in each connected component $C_i$ with exactly two neighbors in $S$ and the rest of vertices of $C_i$ have an unique neighbor in $S$.
\end{rem}

We need also some properties of the set of support vertices on trees reaching the upper bound.

\begin{lem}\label{lem.support.vertices}
Let $T$ be a tree such that $\gamma_{\stackrel{}{11}}(T)=2\gamma (T) -1$, then
\begin{enumerate}
\item The set of support vertices of $G$ is a dominating set.

\item Every support vertex of $G$ is a strong support vertex. Moreover the set of strong support vertices is the unique $\gamma$-code of $T$.
\end{enumerate}
\end{lem}

\begin{proof}
\begin{enumerate}
\item Let $S$ be a $\gamma$-code of $T$ containing all support vertices and suppose on the contrary that there exists $v\in S$ such that is not a support vertex. By hypothesis and using
Lemma~\ref{lem.estructure.S^*}, the $\gamma_{\stackrel{}{11}}$-set associated to $S$ is $S^*=V(T)\setminus L$, with $|S^*|=2\gamma(T)-1$ so it is also a $\gamma_{\stackrel{}{11}}$-code. We are going to construct a smaller $\gamma_{\stackrel{}{11}}$-set of $T$, leading a contradiction.

Denote by $N(v)=\{u_1,\dots,u_s \},\ s\geq 2$, the set of neighbors of $v$. Note that $S$ is independent, so $N(v)\cap S=\emptyset$. Denote by $D_i$ the connected component of $T-S$ containing $u_i,\ i\in \{1,\dots,s \}$. Firstly suppose that each $u_i$ has exactly two neighbors in $S$ (see Figure~\ref{fig.lema3caso1a}(a)). Observe that one of then is vertex $v$ and that $D_i$ contains no leaves of $T$. We define $R=(S^*\setminus \{v\}) \setminus \big(\bigcup_{i=2}^s D_i\big)$ (see Figure~\ref{fig.lema3caso1a}(b)), then $V(T)\setminus R=L\cup  \big(\bigcup_{i=2}^s D_i\big) \cup \{v\}$. Note that any leaf has one neighbor in $R$, also the unique neighbor of $v$ in $R$ is $u_1$ and any vertex in $D_i,\ i\in \{2, \dots, s\}$ is dominated by exactly one vertex in $R$. So $R$ is a $\gamma_{\stackrel{}{11}}$-set of $T$, with smaller cardinal than $S^*$, which is not possible.

Now suppose that vertices $u_1,\dots, u_t$ for some $t\in \{1,\dots,s\}$ have exactly one neighbor in $S$, that must be $v$, and vertices $u_{t+1},\dots ,u_s$ have two neighbors in $S$ (see Figure~\ref{fig.lema3caso1b}(a)). Using that $u_i$ is not a leaf and condition 1 in Lemma~\ref{lem.support.vertices}, we denote by $D_i^*, \ i\in \{1,\dots,t\}$, the connected component of $D_i-\{u_i\}$ containing the unique vertex of $D_i$ with two neighbors in $S$ and let $\widehat{D_i}=D_i-D_i^*$. Then $R=(S^*\setminus \{v\})\setminus \big((\bigcup_{i=2}^t \widehat{D_i})\ \bigcup\  (\bigcup_{j=t+1}^s D_j)\big)$ (see Figure~\ref{fig.lema3caso1b}(b)) is a $\gamma_{\stackrel{}{11}}$-set of $T$, with smaller cardinal than $S^*$, which is not possible.

\item Let $S$ be the $\gamma$-code of $T$ consisting on all support vertices and suppose on the contrary that there exists $v\in S$ which is not a strong support vertex. Again the associated $\gamma_{\stackrel{}{11}}$-set satisfies $S^*=V(T)\setminus L$, with $|S^*|=2\gamma(T)-1$.

Denote by $N(v)=\{u_1,\dots,u_s \},\ s\geq 2$, the set of neighbors of $v$, where $u_1$ is the unique neighbor that is a leave, and by $D_i$ the connected component of $T-S$ containing $u_i,\ i\in \{2,\dots,s \}$. We repeat the construction above, so firstly suppose that each $u_i, i\in\{2,\dots, s\}$ has exactly two neighbors in $S$. Then  $R=\big((S^*\cup \{u_1\})\setminus \{v\}\big) \setminus \bigcup_{i=2}^s D_i$ is a $\gamma_{\stackrel{}{11}}$-set of $T$, with smaller cardinal than $S^*$, which is not possible.

Finally suppose that vertices $u_2,\dots, u_t$ for some $t\in \{2,\dots,s\}$ have exactly one neighbor in $S$, that must be vertex $v$, and vertices $u_{t+1},\dots ,u_s$ has two neighbors in $S$. We use the same notation as above and then the set $R=\big((S^*\cup \{u_1\})\setminus \{v\}\big)\setminus \big((\bigcup_{i=2}^t \widehat{D_i})\ \bigcup\  (\bigcup_{j=t+1}^s D_j)\big)$ is a $\gamma_{\stackrel{}{11}}$-set of $T$, with smaller cardinal than $S^*$, which is not possible.

\end{enumerate}
\end{proof}

\begin{figure}[!ht]
\begin{center}
\includegraphics[width=0.80\textwidth]{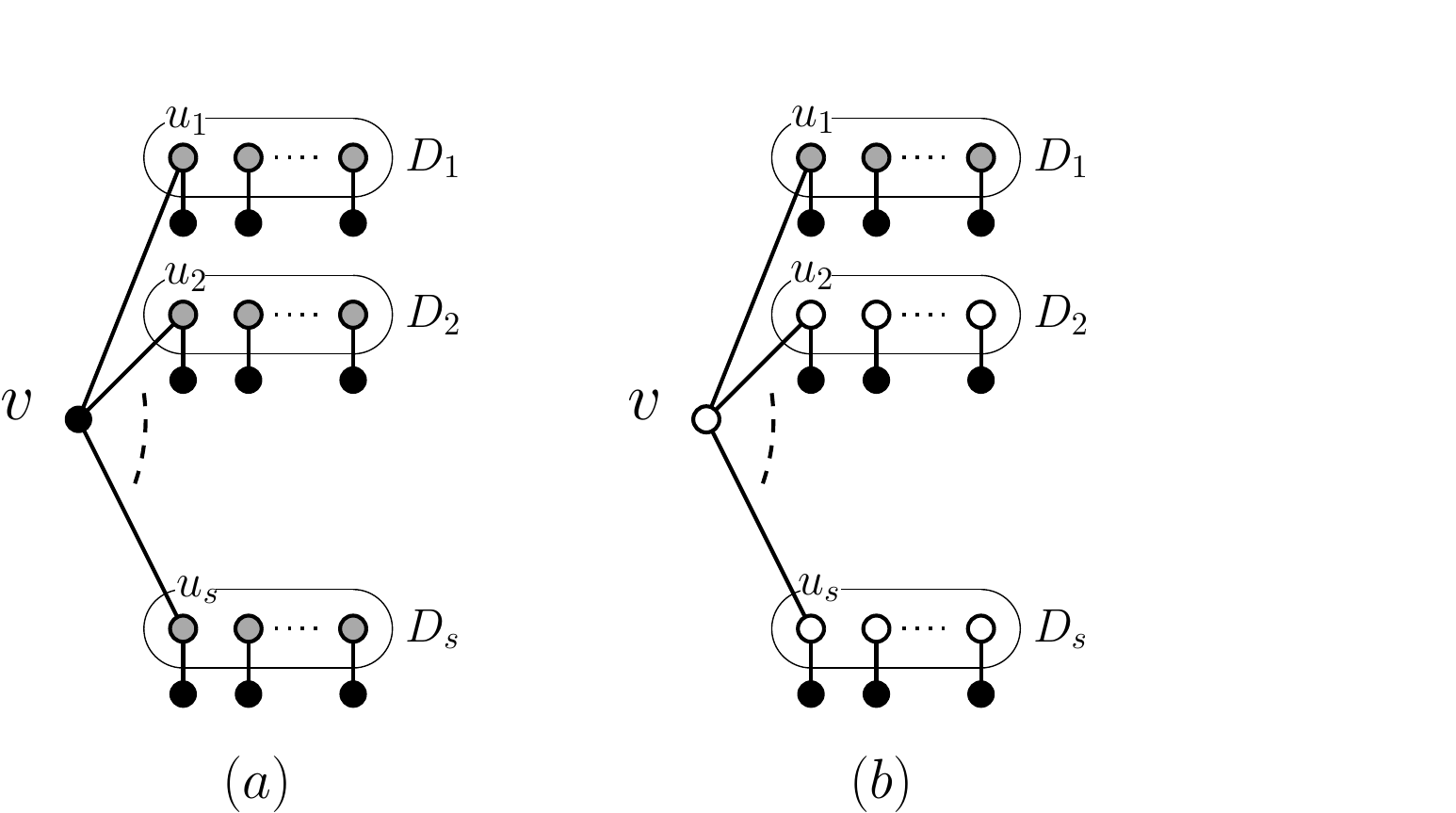}
\caption{ (a) Black vertices are in $S$. Black and gray vertices are in $S^*$.
(b) There is a $\gamma_{11}$-set not containing white vertices.}\label{fig.lema3caso1a}
\end{center}
\end{figure}
\begin{figure}[!ht]
\begin{center}
\includegraphics[width=0.70\textwidth]{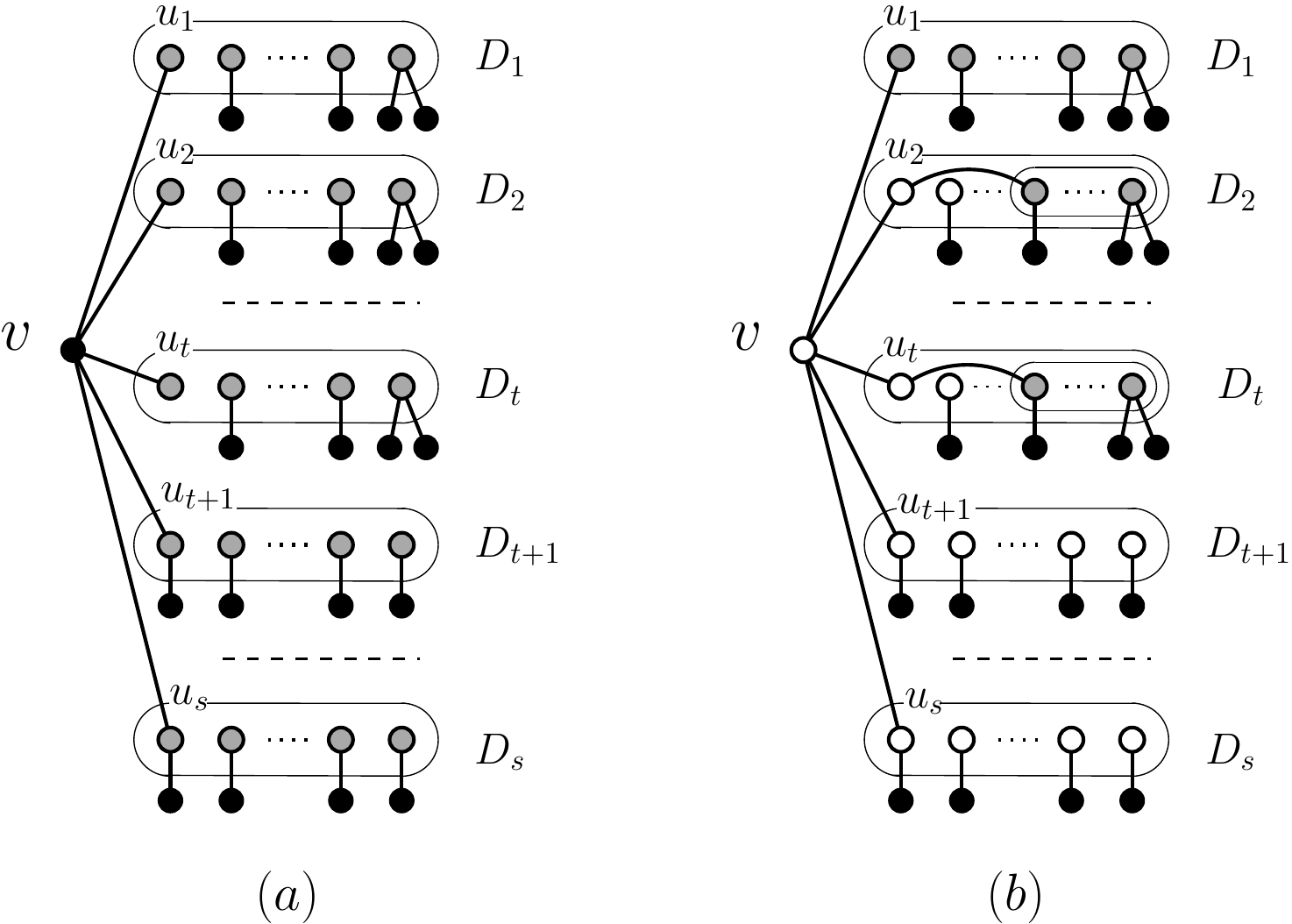}
\caption{(a) Black vertices are in $S$. Black and gray vertices are in $S^*$.
(b) There is a $\gamma_{11}$-set not containing white vertices.}\label{fig.lema3caso1b}
\end{center}
\end{figure}

Now we can characterize trees achieving the upper bound stated in Theorem~\ref{thm.upperbound}, in the case of perfect domination.

\begin{thm}

Let $T$ be a tree. Then $\gamma_{\stackrel{}{11}}(T)=2\gamma (T) -1$ if and only if the following conditions hold:
\begin{enumerate}
\item the set $S$ of strong support vertices of $T$ is an independent dominating set,
\item any connected component C of $T-(S\cup L)$ satisfies $|N(V(C))\cap S|=|V(C)|+1$, that is, every vertex in $C$ has exactly one neighbor in $S$ except one vertex that has two neighbors in $S$.
\end{enumerate}

\end{thm}
\begin{proof}

If $\gamma_{\stackrel{}{11}}(T)=2\gamma (T) -1$, using Lemma~\ref{lem.estructure.S^*} and Lemma~\ref{lem.support.vertices}, it is clear that $T$ satisfies both conditions.

On the other hand suppose that $T$ satisfies conditions 1 and 2. Note that the set $S$ of strong support vertices is the unique $\gamma$-code of $T$. Moreover $S$, and hence its associated
$\gamma_{\stackrel{}{11}}$-set $S^*$, are contained in any $\gamma_{\stackrel{}{11}}$-set of $T$, so $S^*$ is the unique $\gamma_{\stackrel{}{11}}$-code of $T$. By hypothesis $S$ is independent so $E(S:S)=0$ and also any connected
component of $T\setminus (S\cup L)$ has a unique vertex with two neighbors in $S$ so $r=k$ and $\delta_i=1, \ i\in\{1,\dots, r\}$. Finally, using Equation~\ref{eq.final} we obtain
$$2|S|-1=|S^*|+\sum_{i=1}^{r}\delta_i+\sum_{i=r+1}^{k}|V(C_i)|-k+|E(S:S)|= |S^*|+r-r+0=|S^*|$$ and $2\gamma(T)-1=\gamma_{\stackrel{}{11}}(T)$, as desired.

\end{proof}

\vspace{1.5cc}
\noindent{\bf 3.3 Realization result}
\vspace{1cc}

A realization theorem for the inequalities chain $\gamma (T)\le \gamma_{\stackrel{}{11}}(T)\le 2\gamma(T)-1$ is presented. Note that, for every tree $T$ of order $n\ge 3$, Proposition~\ref{equality} and Theorem~\ref{thm.upperbound} give us two possible situations $\gamma (T) = \gamma_{\stackrel{}{11}}(T)\le  n/2$ or $\gamma (T) < \gamma_{\stackrel{}{11}}(T) <n/2$. In the following result we prove that both of them are feasible and parameters $\gamma$ and $\gamma_{\stackrel{}{11}}$ can take every possible value in each case.

\begin{prop}\label{prop.caterpillarrealization}
\begin{enumerate}
  \item Let $a,n$ be integers such that $1\le a$ and $n\ge  2a$, then
there exists a caterpillar $T$ of order $n$ such that $\gamma (T)=\gamma_{\stackrel{}{11}}(T)=a$.
  \item Let $a,b,n$ be integers such that $2\le a<  b\le 2a-1$ and $n> 2b$, then
there exists a caterpillar $T$ of order $n$ such that $\gamma (T)=a$ and $\gamma_{\stackrel{}{11}}(T)=b$.
\end{enumerate}
\end{prop}
\begin{proof}
  \begin{enumerate}
    \item Consider the caterpillar obtained by attaching a leaf to each of the first $a-1$ vertices of a path of order $a$ and $n-2a+1\ge 1$ leaves to the last vertex of the path (see Figure \ref{fig.caterpillar1}). Then the vertices of the path is both a $\gamma$-code and a $\gamma_{\stackrel{}{11}}$-code, and $\gamma (T)=\gamma_{\stackrel{}{11}}(T)=a$.
\begin{figure}[!hbt]
\begin{center}
\includegraphics[width=0.5\textwidth]{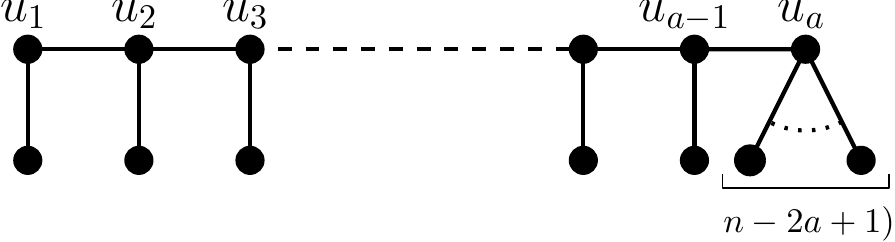}
\caption{$T$ has order $n$, $\gamma (T)=\gamma_{\stackrel{}{11}}(T)=a$.}\label{fig.caterpillar1}
\end{center}
\end{figure}
\item Note that $\gamma (T)=1$ implies $\gamma_{\stackrel{}{11}}(T)=1$, so if both parameter do not agree then  $\gamma (T)\ge 2$.

Using that $1\le b-a\le a-1$, let $P$ be the path of order $b$ with consecutive vertices labeled with
$$u_1,v_1,\dots ,u_{b-a},v_{b-a},u_{b-a+1},u_{b-a+2},\dots ,u_{a}$$
and consider the caterpillar obtained by attaching two leaves to each of the vertices $u_1,u_2,\dots ,u_{b-a}$, one leaf to each of the vertices $u_{b-a+2},u_{b-a+3}, \dots ,u_{a}$ and $n-2b+1$ leaves to vertex $u_{b-a+1}$ (see Figure \ref{fig.caterpillar2}). Since $n-2b+1\ge 2$ we obtain that $\{ u_1,u_2,\dots ,u_a\}$ is a  $\gamma$-code with $a$ vertices and $\{ u_1,u_2,\dots ,u_a\} \cup \{ v_1,\dots ,v_{b-a} \}$ is a $\gamma_{\stackrel{}{11}}$-code with $b$ vertices.
  \end{enumerate}
\end{proof}
\begin{figure}[!hbt]
\begin{center}
\includegraphics[width=0.7\textwidth]{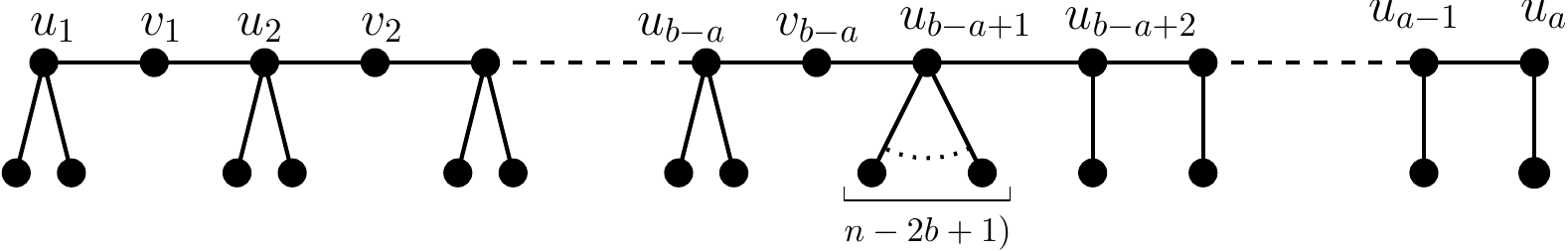}
\caption{$T$ has order $n > 2b$, $a=\gamma (T)<\gamma_{\stackrel{}{11}}(T)=b\le 2a-1$.
%The set $\{ u_1,\dots ,u_a \}$ is  a dominating code and the set $\{ u_1,\dots ,u_a, v_1,\dots , v_{b-a} \}$ is a $\gamma_{\stackrel{}{11}}$-code.
}\label{fig.caterpillar2}
\end{center}
\end{figure}

\vspace{1.5cc}
\begin{center}
{\bf 4. REALIZATION OF THE QP-CHAIN}
\end{center}
\vspace{1cc}

%\noindent{\bf 4.1 Realization}
%\vspace{1cc}

In this Section we present a general realization theorem for the QP-chain and we obtain trees achieving any feasible relationship among the quasiperfect parameters.
We begin with some previous technical results.

\begin{lem}\label{lem.hojas}
If $u$ is a vertex of a graph $G$ with at least $d$ leaves in its neighborhood, then $u$ is in every $\gamma_{\stackrel{}{1h}}$-set, for any $h\in \{ 1,\dots , d-1\}$.
\end{lem}

\begin{proof} Let $h\in \{ 1,\dots , d-1\}$ be and let $S$ be a $\gamma_{\stackrel{}{1h}}$-set of $G$ such that $u\notin S$. Then every leaf adjacent to $u$ must be in $S$, so $u$ has at least $d$ neighbors in $S$, with $d>h$, a contradiction.
\end{proof}

\begin{cor}\label{cor.hojas}
If $G$ is a graph with maximum degree $\Delta$ and $u$ is a  vertex with at least $\Delta -1$ leaves in its neighborhood, then $u$ is in every $\gamma_{\stackrel{}{1h}}$-code, for any $h\in \{ 1,\dots ,  \Delta -2 \}$.
\end{cor}

The following Lemma is trivial.

\begin{lem}\label{lem.suportvertices}
Let $T$ be a tree with maximum degree $\Delta $ and $s$ support vertices. Then $\gamma_{\stackrel{}{1\Delta}} (T) =\gamma (T) \ge s$.
\end{lem}

%\begin{proof}
%For every support vertex $u$ of a tree $T$, at least vertex $u$ or one of the leaves hanging from $u$ must be in a dominating set, otherwise there is no vertex dominating the leaves hanging from $u$.
%\end{proof}

Let $T$ be a tree with maximum degree $\Delta \ge 3$. The next theorem shows that for each inequality of the QP-chain both possibilities, the equality and the strict inequality, are feasible.

\begin{thm}\label{thm.cadena} For any $\Delta \ge 3$, there exists a tree $T$ with maximum degree $\Delta $ satisfying each one of the $2^{\Delta -1}$ possible combinations of the inequalities of the QP-chain
\begin{align*}
  \gamma_{\stackrel{}{11}} (T) \, \ge  \, \gamma_{\stackrel{}{12}} (T) \, \ge \, \gamma_{\stackrel{}{13}} (T) \, \ge \, \dots \, \ge \, \gamma_{\stackrel{}{1(\Delta -1)}} (T) \,  \ge \, \gamma_{\stackrel{}{1\Delta}} (T) =\gamma (T)
\end{align*}
\end{thm}

\begin{proof} Let $\Delta \ge 3$. For all $i\in \{ 1,\dots ,\Delta -1\}$, we write $\circledast_i$ for the symbol  `$=$' or `$>$' in $\gamma_{\stackrel{}{1i}} (T) \, \ge  \, \gamma_{\stackrel{}{1(i+1)}} (T)$.

\begin{enumerate}[]

%%%%%%%%%%%%%%%%%%%%% cas 1
\item

Case 1. If $\circledast_i$ is `$=$' for all $i\in \{ 1,\dots ,\Delta -2\}$.
We distinguish two subcases.

\begin{enumerate}[]
  \item Case 1.1. If $\circledast_{\Delta -1}$ is `$=$'. The star $T=K_{1,\Delta}$ is a tree with maximum degree $\Delta$ satisfying:
 $$\gamma_{\stackrel{}{11}} (T) \, = \, \gamma_{\stackrel{}{12}} (T) \, =
 %\, \gamma_{\stackrel{}{13}} (T) \, =
 \, \dots \, = \, \gamma_{\stackrel{}{1(\Delta -1)}} (T) \,  = \, \gamma_{\stackrel{}{1\Delta}} (T) =\gamma (T)=1.$$

 \item  Case 1.2. If $\circledast_{\Delta -1}$ is `$>$'.
We consider the tree $T$ in Figure~\ref{fig.casIguals_new}. We easily derive from Corollary \ref{cor.hojas} that $\{x_1,\dots ,x_{\Delta} \}$ is a $\gamma$-code and
 $\{u, x_1,\dots ,x_{\Delta} \}$ is a $\gamma_{\stackrel{}{1i}}$-code for any $i$ such that $i<\Delta$. Therefore, $T$ satisfies
 $$\Delta +1=\gamma_{\stackrel{}{11}} (T) \, = \, \gamma_{\stackrel{}{12}} (T) \, =
 %\, \gamma_{\stackrel{}{13}} (T) \, =
 \, \dots \, = \, \gamma_{\stackrel{}{1(\Delta -1)}} (T) \,  > \, \gamma_{\stackrel{}{1\Delta}} (T) =\gamma (T)=\Delta.$$

\begin{figure}[!hbt]
\begin{center}
\includegraphics[width=0.35\textwidth]{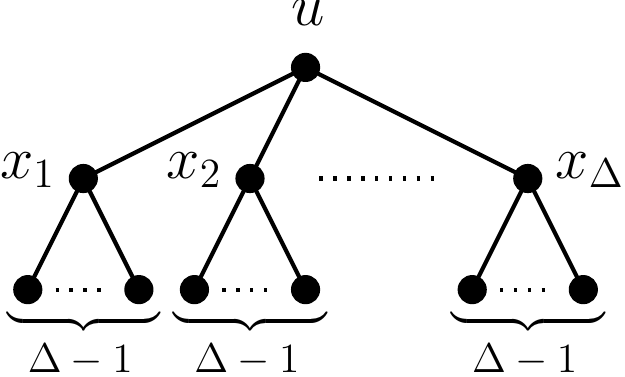}
\caption{Trees illustrating Case 1.2 of Theorem \ref{thm.cadena}.}\label{fig.casIguals_new}
\end{center}
\end{figure}

\end{enumerate}
%%%%%%%%%%%%%%%%%%%%%%%%%% fi cas 1

%%%%%%%%%%%%%%%%%%%%%%% cas 2
\item Case 2. If $\circledast_i$ is `$>$' for some $i\in \{ 1,\dots ,\Delta -2\}$.

If $\Delta=3$, consider the graphs shown in Figure \ref{fig.casDelta3}.
The tree $T$ on the left side satisfies $6=\gamma_{\stackrel{}{11}} (T) \, > \, \gamma_{\stackrel{}{12}} (T) \, = \gamma_{\stackrel{}{1,3}} (T) =\gamma (T)=4$, since support vertices form a $\gamma$-code (and also a $\gamma_{\stackrel{}{12}}$-code and a $\gamma_{\stackrel{}{13}}$-code), and all vertices but the leaves form a $\gamma_{\stackrel{}{11}}$-code.
The tree $T$ on the right side satisfies $\gamma_{\stackrel{}{11}} (T) = 18 \, > \, \gamma_{\stackrel{}{12}} (T) = 12\, > \gamma_{\stackrel{}{1,3}} (T) =\gamma (T)=11$, since support vertices  together with vertex $u$ form a $\gamma$-code (and also a $\gamma_{\stackrel{}{13}}$-code),
support vertices  together with vertices $u$ and $v$ form a  $\gamma_{\stackrel{}{12}}$-code,
and all vertices but the leaves form a $\gamma_{\stackrel{}{11}}$-code.

\begin{figure}[!hbt]
\begin{center}
\includegraphics[width=0.7\textwidth]{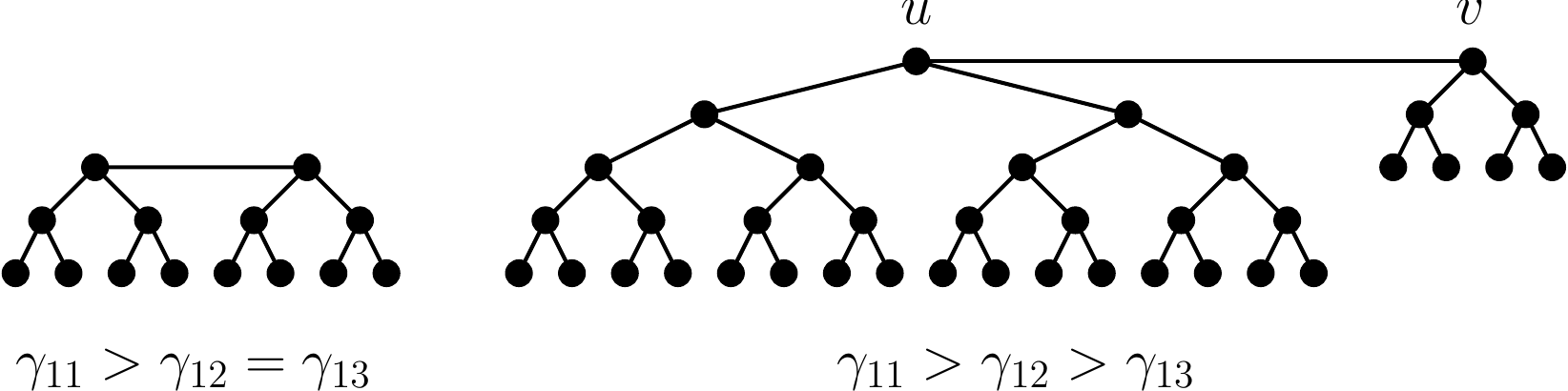}
\caption{Trees illustrating Case 2 of Theorem \ref{thm.cadena} when $\Delta=3$.}\label{fig.casDelta3}
\end{center}
\end{figure}

Now suppose $\Delta \ge 4$. Let
 $$\{ i_1,i_2,\dots ,i_k \} = \{ j \, : \, \gamma_{\stackrel{}{1j}} (T) > \gamma_{\stackrel{}{1(j+1)}} (T) \, , \, j\le \Delta -2 \},$$
where $k\ge 1$ by hypotheses, and assume $1\le i_1<\dots <i_k\le \Delta -2$.
 We distinguish two subcases.

\begin{enumerate}[]
  \item  Case 2.1. If $\circledast_{\Delta -1}$ is `$=$'.

  Consider a path $P$ of length $k+2$ with consecutive vertices labeled $u_{i_1},\dots ,u_{i_k}, v, w$. Attach $i_j$ new vertices to $u_{i_j}$ and $\Delta -1$ leaves to each one of those new vertices. Attach also $\Delta -2$ leaves to vertex $v$.

  For each vertex $x$ of the path $P$, let  $N'(x)$ be the set of vertices of $N(x)$ not belonging to the path $P$. Let $A=\cup_{j=1}^k N'(u_{i_j})$.

  It is clear that $A\cup \{ v \}$ is a $\gamma$-code of $T$, and also a $\gamma_{\stackrel{}{1(\Delta-1)}}$-code. Moreover, $A\cup \{ v \}\cup \{u_{i_j}: h\le j\le k \}$ is a $\gamma_{\stackrel{}{1i}}$-code if $i_{h-1}< i \le i_h$.

   \item Case 2.2. If $\circledast_{\Delta -1}$ is `$>$'.

  Consider the tree constructed in case 2.1 and attach $\Delta -1$ new vertices to $w$ and $\Delta -1$ leaves to each one of those new vertices.

  With the same notations as in Case 2.1, it is easy to verify that $A\cup \{ v \}\cup N'(w)$ is a $\gamma$-code of $T$ and $A\cup \{ v,w \}\cup N'(w)$ is a $\gamma_{\stackrel{}{1(\Delta-1)}}$-code. Moreover,
  $A\cup \{ v,w \}\cup N'(w) \cup  \{u_{i_j}: h\le j\le k \}$ is a $\gamma_{\stackrel{}{1i}}$-code if $i_{h-1}< i \le i_h$.
\end{enumerate}
%%%%%%%%%%%%%%%%%%%%%%%%%% fi cas 2
\end{enumerate}\end{proof}

\begin{figure}[!hbt]
\begin{center}
\includegraphics[width=0.8\textwidth]{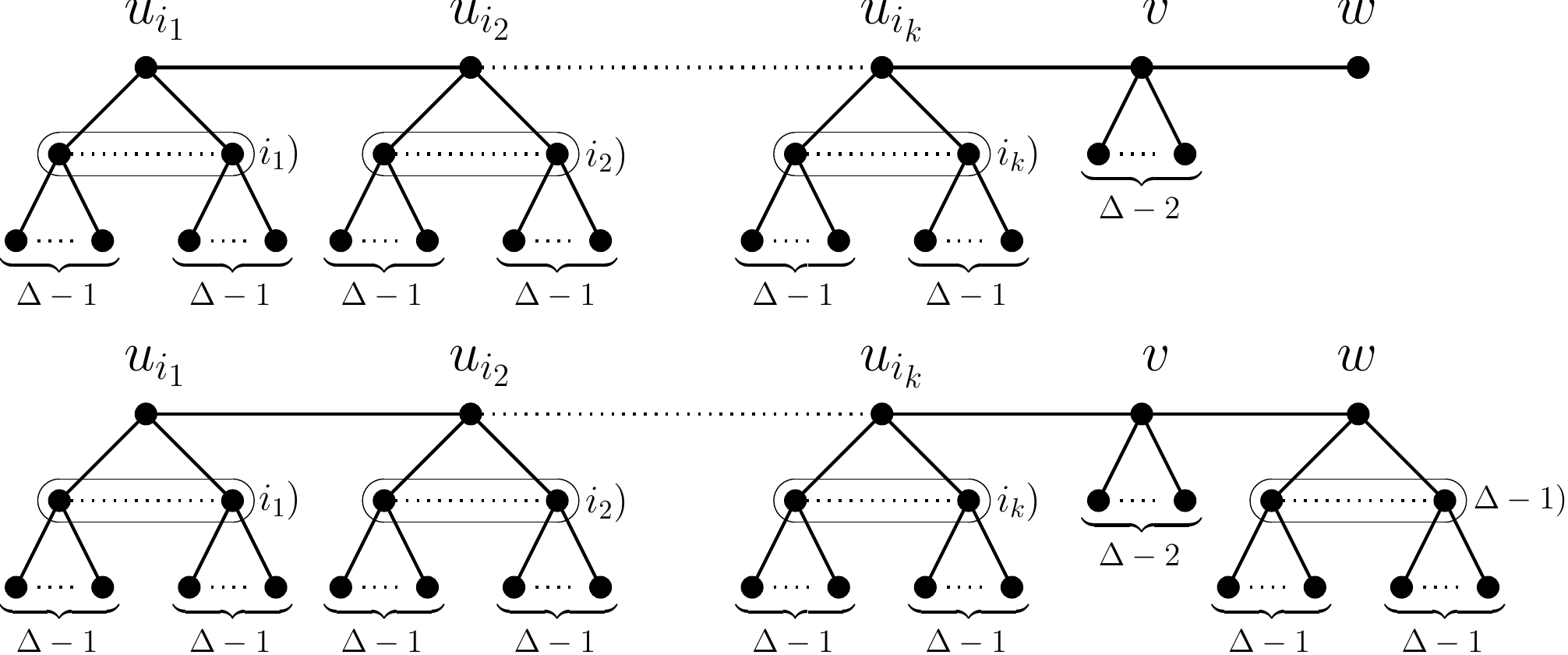}
\caption{Trees illustrating Case 2.1 (above) and Case 2.2 (bottom) of Theorem \ref{thm.cadena}, when $\Delta \ge 4$.}\label{fig.cas2}
\end{center}
\end{figure}

\newpage

\begin{center}
{\bf 5. A LINEAR ALGORITHM FOR TREES}
\end{center}
\vspace{1cc}

The objective of this Section is to devise a linear algorithm for computing $\gamma_{\stackrel{}{1k}}(T)$ for a tree $T$, which answers a question posed in \cite{chhahemc13}, where authors show that the decision problem of determine if a graph has a $\gamma_{\stackrel{}{12}}$-set of cardinality at most $r$ is NP-complete for bipartite graphs. Moreover in \cite{chachevi10} it is shown that the same problem for $\gamma_{\stackrel{}{11}}$-sets is also NP-complete. We will follow the ideas of dynamic programming which appear in~\cite{dhhhml}, where an algorithm to compute the nearly perfect number of a tree in linear time is given.

We will use an operation on rooted trees called \emph{composition}. The composition $(T_1,r_1)\circ (T_2,r_2)$ of two rooted trees is defined as the tree $(T,r_1)$ where $V(T)=V(T_1)\cup V(T_2)$, $E(T)=E(T_1)\cup E(T_2)\cup \{r_1r_2\}$ and its root is $r_1$. The class of rooted trees can be constructed by using this operation and $K_1$ as initial rooted tree with its unique vertex as root.

Let $(T,r)$ be a rooted tree and $S$ a subset of vertices. For a fixed positive integer $k$, we give the next definitions:
\begin{itemize}
\item $S\in A$ if $S$ is a $\gamma_{\stackrel{}{1k}}$-set of $T$, $r\notin S$ and $|N(r)\cap S|=k$.
\item $S\in B$ if $S$ is a $\gamma_{\stackrel{}{1k}}$-set of $T$, $r\notin S$ and $|N(r)\cap S|\leq k-1$.
\item $S\in C$ if $S$ is a $\gamma_{\stackrel{}{1k}}$-set of $T$ and $r\in S$.
\item Finally, $S\in D$ if $S$ is a $\gamma_{\stackrel{}{1k}}$-set of $T-r$ and $N[r]\cap S=\emptyset$.
\end{itemize}

Clearly any $\gamma_{\stackrel{}{1k}}$-set of $T$ belongs to just one of types $A$, $B$ or $C$. The key point in the algorithm is that all the sets of one type can be built in a bottom up form using sets of the above types, which is proved by the next results.

\begin{prop}
\label{th:sets}
Let $(T,r)=(T_1,r)\circ (T_2,r_2)$ be a rooted tree which is the composition of two rooted trees, and let $S$ be a $\gamma_{\stackrel{}{1k}}$-set of $T$. We denote $S_1=S\cap V(T_1)$ and $S_2\cap V(T_2)$. Then:
\begin{enumerate}
\item $S\in A$ if and only if one of the following conditions holds:
    \begin{enumerate}
    \item $|N(r)\cap S_1|=k$, $S_1\in A$ and $S_2\in A\cup B$,
    \item $|N(r)\cap S_1|=k-1$, $S_1\in B$ and $S_2\in C$.
    \end{enumerate}
\item $S\in B$ if and only if one of the following conditions holds:
    \begin{enumerate}
    \item $|N(r)\cap S_1|=k-1$, $S_1\in B$ and $S_2\in A\cup B$,
    \item $|N(r)\cap S_1|\leq k-2$, $S_1\in B$ and $S_2\in A\cup B\cup C$,
    \item $|N(r)\cap S_1|\leq k-2$, $S_1\in D$ and $S_2\in C$.
    \end{enumerate}
\item $S\in C$ if and only if $S_1\in C$ and $S_2\in B\cup C\cup D$.

\end{enumerate}
\end{prop}
\begin{proof}
Here we only prove the sufficiency, since the necessity is a simple exercise.
\begin{enumerate}
\item
    \begin{enumerate}
    \item Assume that $|N(r)\cap S_1|=k$ hence $r_2\notin S$ and let $S\in A$, hence $r_2\notin S$. Then, the edge $rr_2$ joins two vertices not in $S$ and therefore $S_1$ and $S_2$ are $\gamma_{\stackrel{}{1k}}$-sets of $T_1$ and $T_2$ respectively. If $S_2$ is a $\gamma_{\stackrel{}{1k}}$-set where $r_2\notin S_2$, then $S_2\in A\cup B$. On the other hand, all the neighbors of $r$ in $S$ belong to $S_1$, hence $|N(r)\cap S_1|=|N(r)\cap S|=k$ and so $S_1\in A$.
    \item Now assume that $|N(r)\cap S_1|=k-1$. Since $S\in A$, the root $r$ has $k$ neighbors in $S$, so it follows that $r_2\in S$. Note that all dominations in $V(T_2)\setminus S_2$ are exactly the same as in $V(T_2)\setminus S$, so $S_2$ is a $\gamma_{\stackrel{}{1k}}$-set of $T_2$ and therefore $S_2\in C$. On the other hand, although $r$ is not dominated by $r_2$ in $T_1$, it has $k-1$ neighbors in $S_1$ so $S_1$ is a $\gamma_{\stackrel{}{1k}}$-set of $T_1$ and therefore $S_1\in B$.
    \end{enumerate}
\item
    \begin{enumerate}
    \item Suppose that $|N(r)\cap S_1|=k-1$ and $S\in B$. So the $k-1$ neighbors of $r$ belong to $S_1$ and thus $r_2\notin S_2$. Consequently, both $S_1$ and $S_2$ are $\gamma_{\stackrel{}{1k}}$-set of $T_1$ and $T_2$ respectively. Hence $S_1\in B$ and $S_2\in A\cup B$.
    \item Let $|N(r)\cap S_1|\leq k-2$ and suppose that $1\leq |N(r)\cap S_1|$ then $S_1$ is a $\gamma_{\stackrel{}{1k}}$-set of $T_1$ and therefore $S_1\in B$. Clearly $S_2$ is a $\gamma_{\stackrel{}{1k}}$-set of $T_2$ therefore $S_2\in A\cup B\cup C$.

    \item Now if $|N(r)\cap S_1|\leq k-2$ and $|N(r)\cap S_1|=0$ then $S_1$ is a $\gamma_{\stackrel{}{1k}}$-set of $T_1-r$, hence $S_1\in D$. In this case $r_2\in S_2$ and $S_2$ is a $\gamma_{\stackrel{}{1k}}$-set of $T_2$ with $S_2\in  C$.
    \end{enumerate}
\item Let $S\in C$. Any vertex in $V(T_1)\setminus S_1$ is dominated by the same vertices as in $S$, so $S_1$ is a $\gamma_{\stackrel{}{1k}}$-set of $T_1$ and $S_1\in C$. However $r_2$ may or may not belong to $S$. In the former case, we can reason analogously as above and conclude that $S_2\in C$. In the later case, all the vertices in $V(T_2)\setminus S_2$ except $r_2$ are dominated by at least one and at most $k$ vertices in $S_2$, and $r_2$ by at most $k-1$ vertices. If $r_2$ is dominated by some vertex in $S_2$ then $S_2$ is a $\gamma_{\stackrel{}{1k}}$-set in $B$. Otherwise, $S_2$ is a $\gamma_{\stackrel{}{1k}}$-set of $T_2-r_2$ in $D$.

\end{enumerate}
\end{proof}

\begin{prop}
Let $(T,r)=(T_1,r)\circ (T_2,r_2)$ be a rooted tree which is the composition of two rooted trees, and let $S$ be a subset of its vertices. We denote $S_1=S\cap V(T_1)$ and $S_2\cap V(T_2)$. Then $S\in D$ if and only if $S_1\in D$ and $S_2\in A\cup B$.
\end{prop}

\begin{proof}
We only prove the sufficiency. Suppose $S\in D$, i.e., all the vertices in $T$ except $r$ are dominated by at least one and at most $k$ vertices in $S$. Therefore, $S_1$ inherits this property for $T_1$ and $S_1\in D$. On the other hand, $S_2$ should be a $\gamma_{\stackrel{}{1k}}$-set for $T_2$. Since $r$ is not dominated in $S$, the vertex $r_2$ does not belong to $S$, hence $S_2\in A\cup B$.
\end{proof}

In the algorithm, we assume that the vertices of the tree have been numbered from $1$ to $n$ such that all vertices have a greater number than its parent. The tree is stored in the array \texttt{Parent} in which any vertex \texttt{i} points to the location of its parent. At any time of the execution of the second loop, the four variables called \texttt{a(i), b(i), c(i)} and \texttt{d(i)} store the minimum cardinalities of sets of type $A, B, C$ and $D$ for the trees having \texttt{i} as root and previously processed vertices. Any of this variables might be infinite due to either it is not possible to find such sets or there exists a set of different type with the same cardinality. Those variables are initialized with the values corresponding to $K_1$. It is not difficult to modify the algorithm in order to keep track of the final $\gamma_{\stackrel{}{1k}}$-code.

It is necessary to use a fifth variable \texttt{z(i)} to decide between the two possible options for the cardinal of type $B$ sets given in Theorem~\ref{th:sets}. Specifically, \texttt{z(i)} is defined as $|N(i)\cap S|$ where $|S|$ has finite cardinality \texttt{b(i)}, and to keep internal consistency \texttt{z(i)} will be $\infty$ whenever \texttt{b(i)} is infinite.

Note that any $\gamma_{\stackrel{}{1k}}$-set of a rooted tree $T$ is in $A\cup B\cup C$, thus the resulting $\gamma_{\stackrel{}{1k}}$-code will have as cardinality the minimum value among \texttt{a(1), b(1), c(1)}.
\medskip

\noindent \texttt{Algorithm $\gamma_{\stackrel{}{1k}}$ for trees}\\
\noindent \hspace*{5 mm}\texttt{Input:the parent array Parent[1\ldots n] for any tree T}\\
\noindent \hspace*{5 mm}\texttt{Output: $\gamma_{\stackrel{}{1k}}$(T)}\\
\noindent \texttt{begin}\\
\noindent \hspace*{5 mm}\texttt{for i:=1\ldots n do}\\
\noindent \hspace*{10 mm}\texttt{initialize a(i):=$\infty$; b(i):=$\infty$; c(i):=1; d(i):=0; z(i):=$\infty$}\\
\noindent \hspace*{5 mm}\texttt{od}\\
\noindent \hspace*{5 mm}\texttt{for i:=n\ldots 2 do}\\
\noindent \hspace*{10 mm}\texttt{j:=Parent[i]; z:=z(j)}\\
\noindent \hspace*{10 mm}\texttt{a:=min(a(j)+a(i),a(j)+b(i));}\\
\noindent \hspace*{10 mm}\texttt{if a>b(j)+c(i) and z(j)==k-1 then}\\
\noindent \hspace*{15 mm}\texttt{a:=b(j)+c(i)}\\
\noindent \hspace*{10 mm}\texttt{fi}\\
\noindent \hspace*{10 mm}\texttt{b:=min(b(j)+a(i),b(j)+b(i));}\\
\noindent \hspace*{10 mm}\texttt{if b>d(j)+c(i) then}\\
\noindent \hspace*{15 mm}\texttt{b:=d(j)+c(i);}\\
\noindent \hspace*{15 mm}\texttt{z:=1}\\
\noindent \hspace*{10 mm}\texttt{fi}\\
\noindent \hspace*{10 mm}\texttt{if b>b(j)+c(i) and z(j)$\leq$ k-2 then}\\
\noindent \hspace*{15 mm}\texttt{b:=b(j)+c(i);}\\
\noindent \hspace*{15 mm}\texttt{z:=z(j)+1}\\
\noindent \hspace*{10 mm}\texttt{fi}\\
\noindent \hspace*{10 mm}\texttt{if b==$\infty$ then}\\
\noindent \hspace*{15 mm}\texttt{z:=$\infty$}\\
\noindent \hspace*{10 mm}\texttt{fi}\\
\noindent \hspace*{10 mm}\texttt{c:=min(c(j)+b(i),c(j)+c(i),c(j)+d(i));}\\
\noindent \hspace*{10 mm}\texttt{d:=min(d(j)+a(i),d(j)+b(i))}\\
\noindent \hspace*{10 mm}\texttt{a(j):=a; b(j):=b; c(j):=c; d(j):=d; z(j):=z;}\\
\noindent \hspace*{5 mm}\texttt{od}\\
\noindent \hspace*{5 mm}\texttt{$\gamma_{\stackrel{}{1k}}$(T):=min(a(1),b(1),c(1));}\\
\noindent \texttt{end.}

In Figure~\ref{fig:tree_example} is shown an example of the output of the algorithm for $k=3$. The vertices of the tree are labelled as in the initial order. It is also shown the final values of the variables \texttt{a(i), b(i), c(i), d(i)} and \texttt{z(i)} for the internal vertices.

\begin{figure}[htbp]
\begin{center}
\includegraphics[width=6.5cm]{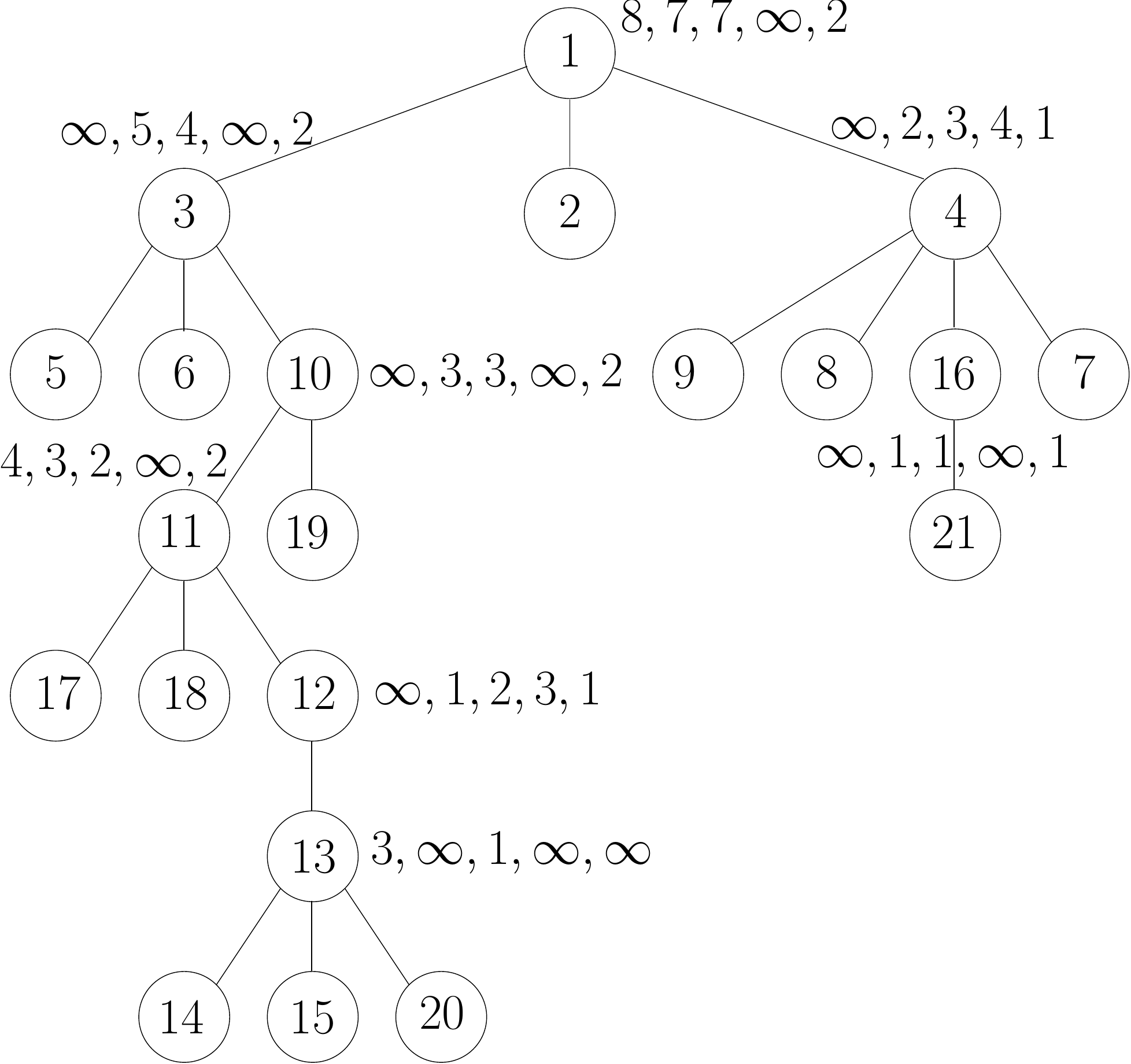}
\end{center}
\caption{An example of the output of the algorithm $\gamma_{\stackrel{}{1k}}$ for trees and $k=3$. For instance, \texttt{a(10)=$\infty$, b(10)=3, c(10)=3, d(10)=$\infty$} and \texttt{z(10)=2}.}
\label{fig:tree_example}
\end{figure}

\begin{thm}
For any tree $T$ with $n$ vertices, $\gamma_{\stackrel{}{1k}}(T)$ can be computed in linear time.
\end{thm}
\begin{proof}
Clearly, the second loop is iterated $n$ times and the operations within the loop can be computed in constant time.
\end{proof}

\noindent
{\bf Acknowledgements.} Authors are partially support by MTM2012-30951/FEDER, MTM2011-28800-C02-01, Gen. Cat. DGR 2014SGR46 , Gen. Cat. DGR 2009SGR1387 and JA-FQM 305.

%\newpage
%%%%%%%%%%%

\vspace{1cc}

%\newpage

{\small
\noindent
Department of Mathematics\\
University of Almer\'ia. Almer\'ia, Spain\\
E-mail: \url{jcaceres@ual.es, mpuertas@ual.es} \\

\noindent
Department of Applied Mathematics I\\
Politechnical University of Catalu\~na. Barcelona, Spain\\
E-mail: \url{carmen.hernando@upc.edu}\\

\noindent
Department of Applied Mathematics II,\\
Politechnical University of Catalu\~na. Barcelona, Spain \\
E-mail: \url{merce.mora@upc.edu}\\

\noindent
Department of Applied Mathematics III,\\
Politechnical University of Catalu\~na. Barcelona, Spain \\
E-mail: \url{ignacio.m.pelayo@upc.edu}\\

}

\end{document}